\documentclass{article}
\usepackage{amsmath,amsthm,amscd,amssymb}
\usepackage[mathscr]{eucal}
\usepackage{amssymb}
\usepackage{latexsym}

\theoremstyle{definition}
\newtheorem{Def}{Definition}[section]
\newtheorem{Thm}[Def]{Theorem}
\newtheorem{Prop}[Def]{Proposition}
\newtheorem{Rem}[Def]{Remark}

\newtheorem{Cor}[Def]{Corollary}

\newtheorem{Lem}[Def]{Lemma}

\numberwithin{equation}{section}

\title{Supplement to a Shimura's theorem on Eisenstein series}

\date{}
\begin{document}
\maketitle
\author{Shoyu Nagaoka}
\footnote{\noindent
Shoyu Nagaoka\\
Dept. Mathematics, Yamato University\\
Suita, Osaka 564-0082,Japan\\
shoyu1122.sn@gmail.com
\\
}
\noindent
\begin{abstract}
\noindent
Shimura studied the analytic properties of the non-holomorphic Siegel
Eisenstein series and derived a residue formula.
Herein, we provide a refinement of his result for several types of Eisenstein series.
\end{abstract}

\section{Introduction}
\label{intro}
The Eisenstein series is an important concept in the field of automorphic forms
and has been studied by various researchers.

In \cite{Sh}, Shimura considered the Eisenstein series for various 
types of groups and
extensively studied their analytic properties. For example, he demonstrated
that the residue of an Eisenstein series at a certain point is a power of $\pi$ times a
modular form with rational Fourier coefficients (see theorem below).

In this paper, we show that this modular form can be easily
specified explicitly using the functional equation of the corresponding Eisenstein
series.

Subsequently, we explain the result in the case of the Eisenstein series for the
Siegel modular group $\varGamma_n:=\text{Sp}_n(\mathbb{Z})$.

For $n\in\mathbb{Z}_{>0}$ and $k\in 2\mathbb{Z}_{\geq 0}$,  we set
$$
E_k^{(n)}(Z,s):=
\text{det}(\text{Im}(Z))^s
\sum_{\binom{\,*\,\,*\,}{C\,D}\in  (\varGamma_n\cap P_n)\backslash \varGamma_n }
\text{det}(CZ+D)^{-k}|\text{det}(CZ+D)|^{-2s}.
$$
Here, $Z$ is a variable of the Siegel upper half space $\mathbb{H}_n$, of degree $n$,
$s$ is a complex variable, and $P_n$ is a parabolic subgroup 
$\{\binom{A\,B}{C\,D}\in \text{Sp}_n(\mathbb{R})\mid\,C=0_n\}$.
The abovementioned series converges with locally and uniformly on
$$
\left\{ (Z,s)\;\big{|}\; Z\in\mathbb{H}_n,\;\text{Re}(s)>(n+1-k)/2\;\right\}.
$$
As is well known from the Langlands theory, $E_k^{(n)}(Z,s)$ has a meromorphic
continuation to the whole complex $s$-plane. In \cite{Sh}, Shimura studied the
analytic properties for various types of Eisenstein series, including
the type mentioned herein.

His result in this case can be stated as follows.
\vspace{2mm}
\\
\textbf{Theorem.} (Shimura \cite{Sh}, {\scshape Proposition 10.3})
{\it The Eisenstein series $E_{\frac{n-1}{2}}^{(n)}(Z,s)$ has at most a simple
pole at $s=1$. The residue is $\pi^{-n}$ times a modular form $f$ of weight
$\frac{n-1}{2}$ for $\varGamma_n$ with rational Fourier coefficients.}
\vspace{2mm}
\\
One of the aims in this study is to specify the aforementioned modular form $f$.

The result in this case is as follows.
\vspace{2mm}
\\
\textbf{Theorem\; (SP case)} {\it The residue of $E_{\frac{n-1}{2}}^{(n)}(Z,s)$ at $s=1$
is expressed as
$$
\mathop{{\rm Res}}_{s=1}E_{\frac{n-1}{2}}^{(n)}(Z,s)=\pi^{-n}\cdot c_n\,E_{\frac{n-1}{2}}^{(n)}(Z,0),
$$
where the constant $c_n$ is explicitly defined as
$$
c_n=
 (-1)^{\frac{n-1}{4}}2^{-2n-2}\cdot 3^{-1}
    \frac{\left(\tfrac{n+3}{2}\right)!(n+1)!}
           {\left(\tfrac{n-1}{2}\right)!} 
     \frac{B_{\frac{n-1}{2}}} {B_{\frac{n+3}{2}}B_{n+1}B_{n-1}}.
$$
Here, $B_m$ is the $m$-th Bernoulli number.}
\vspace{4mm}
\\
This theorem asserts that the modular form $f$ in Shimura's theorem is 
$c_n\,E_{\frac{n-1}{2}}^{(n)}(Z,0)$ in this case.
\vspace{3mm}
\\
In the latter part of the paper, the
author provides similar results for the Eisenstein series of other types,
namely, the Hermitian Eisenstein series, Eisenstein series on a quaternion half-space,
and Eisenstein series on an exceptional domain.
\vspace{2mm}
\\
Some of the results presented herein have been reported in \cite{Nagaoka2},
albeit without proof. This paper provides the proof, including its amendments. 
\section{Siegel Eisenstein series}
\subsection{Siegel modular forms}
Let $\varGamma_n=\text{Sp}_n(\mathbb{Z})$ be the Siegel modular group of degree $n$
and $M_k(\varGamma_n)$ be the space of Siegel modular forms of weight $k$ for
$\varGamma_n$. Any element $F$ in $M_k(\varGamma_n)$ has a Fourier expansion
of the form
$$
F(Z)=\sum_{0\leq T\in \Lambda_n} a_F(T)\text{exp}(2\pi\sqrt{-1}\text{tr}(TZ)),\;
Z\in \mathbb{H}_n,
$$
where
$$
\Lambda_n:=\{\;T=(t_{ij})\in \text{Sym}_n(\mathbb{Q})\;|\; t_{ii},\,2t_{ij}\in\mathbb{Z}\,\}.
$$
For a subring, $R\subset \mathbb{C}$, we denote by $M_k(\varGamma_n)_R$
the space comprising
modular forms $F$ in $M_k(\varGamma_n)$,
whose Fourier coefficients $a_F(T)$ lie in $R$.

We refer to series $E_k^{(n)}(Z,s)$ (defined in the Introduction section) as the {\it Siegel Eisenstein series}.
If $k>n+1$, the Siegel Eisenstein series $E_k^{(n)}(Z,0)$ is holomorphic in $Z$, and it
is a typical example of an element in $M_k(\varGamma_n)_{\mathbb{Q}}$.
\subsection{Functional equation of Siegel Eisenstein series}
For $n\in\mathbb{Z}_{>0}$ and $k\in 2\mathbb{Z}_{\geq 0}$, we define
function $f_{n,k}(s)$ as
\begin{equation}
f_{n,k}(s):=\frac{\Gamma_n(s+\frac{k}{2})}{\Gamma_n(s)}\cdot \xi (2s)
                 \prod_{j=1}^{[\frac{n}{2}]}\xi (4s-2j),
\end{equation}
where
$$
\Gamma_n(s):=\pi^{\frac{n(n-1)}{4}}\prod_{j=0}^{n-1}\Gamma\left(s-\tfrac{j}{2}\right),
\qquad
\xi (s):=\pi^{-\frac{s}{2}}\Gamma\left(\tfrac{s}{2}\right)\zeta(s)=\xi (1-s).
$$
The Siegel Eisenstein series $E_k^{(n)}(Z,s)$ has a functional equation of the form
\begin{equation}
\label{FE}
E_k^{(n)}(Z,s)
=\frac{f_{n,k}\left(\frac{n+1}{2}-\frac{k}{2}-s\right)}{f_{n,k}\left(s+\frac{k}{2}\right)}
E_k^{(n)}\left(Z,\tfrac{n+1}{2}-k-s\right)
\end{equation}
(e.g., cf. \cite{Miz}).
\subsection{Analytic property of $\boldsymbol{E_k^{(n)}(Z,s)}$ at $\boldsymbol{s=0}$}
The analytic property of $E_k^{(n)}(Z,s)$ at $s=0$ has been studied by
Shimura \cite{Sh}, Weissauer \cite{Weis}, and subsequently by Haruki \cite{Haruki}.
\begin{Thm} (Weissauer \cite[Satz 17]{Weis}, Haruki \cite{Haruki})\\
\label{Wei}
\noindent
{\it \!\!{\rm (1)}\; If $k>0$, $E_k^{(n)}(Z,s)$ is holomorphic in $s$ at $s=0$.
\\
In case {\rm (1)}, we define $E_k^{(n)}(Z):=E_k^{(n)}(Z,0)$.\\
{\rm (2)}\;  $E_k^{(n)}(Z)$ is holomorphic except for the two cases, i.e.,
$$
k=\tfrac{n+2}{2},\;\tfrac{n+3}{2} \equiv 2 \pmod{4}.
$$
{\rm (3)}\; In the holomorphic case, $E_k^{(n)}(Z)$ has rational Fourier coefficients.}
\end{Thm}
\begin{Rem}
\label{vanish}
If we take the theory
of singular modular forms into account, it turns out that
$E_k^{(n)}(Z)$ vanishes when $k<\frac{n}{2}$ and
$k \equiv 2 \pmod{4}$.
\end{Rem}
\subsection{Shimura's result}
As stated in the Introduction section, Shimura proved the following result.
\begin{Thm} (Shimura, \cite[{\scshape Proposition 10.3}]{Sh})
\label{ShimuraMain}
{\it 
Assume that $\tfrac{n-1}{2}\in 2\mathbb{Z}_{>0}$. 
The Eisenstein series
 $E_{\frac{n-1}{2}}^{(n)}(Z,s)$ has at most a simple pole
at $s=1$. The residue is written as
$$
\mathop{{\rm Res}}_{s=1}E_{\frac{n-1}{2}}^{(n)}(Z,s)=\pi^{-n}\cdot f
$$
with some $f\in M_{\frac{n-1}{2}}(\varGamma_n)_{\mathbb{Q}}$.}
\end{Thm}
\begin{Rem}
The Eisenstein series that Shimura considered in \cite{Sh} is
$$
\mathcal{E}_k^{(n)}(Z,s) =\sum \text{det}(CZ+D)^{-k}|\text{det}(CZ+D)|^{-s},
$$
i.e.,
$$
\mathcal{E}_k^{(n)}(Z,s)=
\text{det}(\text{Im}(Z))^{-\frac{s}{2}}E_k^{(n)}\left(Z,\frac{s}{2}\right)
$$
in our notation. Theorem \ref{ShimuraMain} is a translation 
of his original statement of {\scshape Proposition} 10.3 in \cite{Sh}.
\end{Rem}
\subsection{Main result for Siegel Eisenstein series}
\subsubsection{Residue of Siegel Eisenstein series}
To refine Shimura's result, we prove the following theorem.
\begin{Thm}
\label{Residue1}{\it 
Let $n$ and $m$ be integers that satisfy $n>m\geq 1$ and 
$\frac{n-m}{2}\in 2\mathbb{Z}_{>0}$.
Subsequently, the Siegel Eisenstein series $E_{\frac{n-m}{2}}^{(n)}(Z,s)$ has at most a simple
pole at $s=\frac{m+1}{2}$, and
\begin{equation}
\label{Res}
\mathop{{\rm Res}}_{s=\frac{m+1}{2}}E_{\frac{n-m}{2}}^{(n)}(Z,s)=c_{n,m}\,E_{\frac{n-m}{2}}^{(n)}(Z,0),
\end{equation}
where
\begin{align}
\label{explicitcnm} 
c_{n,m} & =(-1)^{[\frac{m+3}{2}]}\cdot\tfrac{1}{4}\cdot
      \frac{\Gamma_{\frac{n-m}{2}}\left(\tfrac{n-m}{2}\right)}{\Gamma_{\frac{n-m}{2}}\left(\tfrac{n+1}{2}\right)}
\prod_{j=0}^{\frac{n-m-4}{4}}\frac{\Gamma(\tfrac{1}{2}-j)}{\Gamma(-[\tfrac{m}{2}]-\tfrac{1}{2}-j)}
    \cdot \frac{([\frac{m+1}{2}]+j)!}{j!}
 \\ \nonumber
 & \cdot \frac{\xi\left(\tfrac{n-m}{2}\right)}{\xi\left(\tfrac{n+m+2}{2}\right)}\,
        \frac{\prod_{\substack{1\leq j\leq [\frac{n}{2}] \\ j\ne\frac{n-m}{2}}}\xi(n-m-2j)}
             {\prod_{j=1}^{[\frac{n}{2}]}\xi(n+m+2-2j)}.
\end{align}
}
\end{Thm}
When $\frac{n-m}{2} \equiv 2 \pmod{4}$, $E_{\frac{n-m}{2}}^{(n)}(Z,0)$ vanishes
identically because of the theory of singular modular forms (Remark \ref{vanish}).
Hence, we have the following corollary.
\begin{Cor}{\it
If $\frac{n-m}{2} \equiv 2 \pmod{4}$, $E_{\frac{n-m}{2}}^{(n)}(Z,s)$ is holomorphic at
$s=\frac{m+1}{2}$.}
\end{Cor}
Next, we prove Theorem \ref{Residue1}.
First, we remark that the holomorphy of
$E_{\frac{n-m}{2}}^{(n)}(Z,s)$ at $s=0$ is guaranteed by 
Theorem \ref{Wei},  (2).
The functional equation of $E_{\frac{n-m}{2}}^{(n)}(Z,s)$ can be written as
\begin{equation}
\label{FEQ}
E_{\frac{n-m}{2}}^{(n)}(Z,s)=F_{n,m}(s)\,E_{\frac{n-m}{2}}^{(n)}\left(Z,\tfrac{m+1}{2}-s\right),
\end{equation}
where
\begin{align}
& F_{n,m}(s)=\gamma_{n,m}(s)\,\xi_{n,m}(s),   \label{Fnm} \\
& \gamma_{n,m}(s):=\frac{\Gamma_n\left(\frac{n+1}{2}-s\right)}
                                     {\Gamma_n\left(\frac{n+m+2}{4}-s\right)}\cdot
                                     \frac{\Gamma_n\left(s+\frac{n-m}{4}\right)}
                                     {\Gamma_n\left(s+\frac{n-m}{2}\right)},  \nonumber\\
& 
\xi_{n,m}(s):=\frac{\xi\left(\frac{n+m+2}{2}-2s\right)}{\xi\left(2s+\frac{n-m}{2}\right)}
                      \prod_{j=1}^{\left[\frac{n}{2}\right]}
                      \frac{\xi((n+m+2)-4s-2j)}{\xi (4s+(n-m)-2j)}. \nonumber
\end{align}
To prove Theorem \ref{Residue1}, it is sufficient to demonstrate that 
\begin{equation}
\label{main}
F_{n,m}(s) \text{\;has a simple pole at}\; s=\tfrac{m+1}{2},\;
\text{and the residue is}\;c_{n,m}
\end{equation}
because $E_{\frac{n-m}{2}}^{(n)}(Z,s)$ is holomorphic at $s=0$ (cf. Theorem \ref{Wei}). 
Thus, if we set
$$
c_{n,m}=\mathop{{\rm Res}}_{s=\frac{m+1}{2}}\,F_{n,m}(s),
$$
then the theorem is proved.
Here, we show the proof of (\ref{main}).
\subsubsection{Analysis of $\gamma$-factor $\boldsymbol{\gamma_{n,m}(s)}$}
\begin{Prop}
\label{P1}{\it 
Function $\gamma_{n,m}(s)$ in {\rm (\ref{Fnm})} is holomorphic at $s=\frac{m+1}{2}$ and
\begin{equation}
\label{gammanm}
\gamma_{n,m}\left(\tfrac{m+1}{2}\right)
=
(-1)^{[\frac{m+3}{2}]}
\cdot
\frac{\Gamma_{\frac{n-m}{2}}\left(\tfrac{n-m}{2}\right)}{\Gamma_{\frac{n-m}{2}}\left(\tfrac{n+1}{2}\right)}
\prod_{j=0}^{\frac{n-m-4}{4}}\frac{\Gamma(\tfrac{1}{2}-j)}{\Gamma(-[\tfrac{m}{2}]-\tfrac{1}{2}-j)}
    \cdot \frac{([\frac{m+1}{2}]+j)!}{j!}.
\end{equation}
}
\end{Prop}
\begin{proof}
The factors that appear in the definition of $\gamma_{n,m}(s)$
can be simplified using the following formulas.
\begin{Lem}
(Cancellation law)
\label{cancel}
$$
\frac{\Gamma_n\left(\frac{n+1}{2}-s\right)}
       {\Gamma_n\left(\frac{n+m+2}{4}-s\right)}
=
\frac{\Gamma_{\frac{n-m}{2}}\left(\frac{n+1}{2}-s\right)}
       {\Gamma_{\frac{n-m}{2}}\left(\frac{1}{2}-s\right)},\qquad
\frac{\Gamma_n\left(s+\frac{n-m}{4}\right)}
       {\Gamma_n\left(s+\frac{n-m}{2}\right)}
=
\frac{\Gamma_{\frac{n-m}{2}}\left(s-\frac{m}{2}\right)}
       {\Gamma_{\frac{n-m}{2}}\left(s+\frac{n-m}{2}\right)}.
$$
\end{Lem}
This lemma results in the decomposition of $\gamma_{n,m}(s)$ as follows:
\begin{align}
\gamma_{n,m}(s)= & \gamma_{n,m}^{(\text{I})}(s)\cdot \gamma_{n,m}^{(\text{II})}(s), \label{gamma}\\
      & \gamma_{n,m}^{(\text{I})}(s)=
      \frac{\Gamma_{\frac{n-m}{2}}\left(\frac{n+1}{2}-s\right)}
       {\Gamma_{\frac{n-m}{2}}\left(s+\frac{n-m}{2}\right)}, \quad
      \gamma_{n,m}^{(\text{II})}(s)=
        \frac{\Gamma_{\frac{n-m}{2}}\left(s-\frac{m}{2}\right)}
       {\Gamma_{\frac{n-m}{2}}\left(\frac{1}{2}-s\right)}. \nonumber
\end{align}
\textbf{Calculation of} $\boldsymbol{\gamma_{n,m}^{(\text{I})}(s)}$:
\vspace{1mm}
\\
Functions $\Gamma_{\frac{n-m}{2}}\left(\frac{n+1}{2}-s\right)$ and $\Gamma_{\frac{n-m}{2}}\left(s+\frac{n-m}{2}\right)$
are holomorphic at $s=\frac{m+1}{2}$ and thier values at $s=\frac{m+1}{2}$ are non-zero.
Hence, $\gamma_{n,m}^{(\text{I})}(s)$ is
holomorphic at $s=\frac{m+1}{2}$, and 
\begin{equation}
\label{gammanmI}
\gamma_{n,m}^{(\text{I})}\left( \tfrac{m+1}{2}\right)
=
\frac{\Gamma_{\frac{n-m}{2}}\left(\tfrac{n-m}{2}\right)}{\Gamma_{\frac{n-m}{2}}\left(\tfrac{n+1}{2}\right)}.
\end{equation}
\textbf{Calculation of} $\boldsymbol{\gamma_{n,m}^{(\text{II})}(s)}$:
\vspace{1mm}
\\
Direct calculation shows
\begin{align}
\lim_{s\to\frac{m+1}{2}}
\gamma_{n,m}^{(\text{II})}(s) & =\lim_{s\to\frac{m+1}{2}}
        \frac{\Gamma_{\frac{n-m}{2}}\left(s-\frac{m}{2}\right)}
       {\Gamma_{\frac{n-m}{2}}\left(\frac{1}{2}-s\right)}
       =
       \lim_{t\to 0}\frac{\Gamma_{\frac{n-m}{2}}(t+\frac{1}{2})}{\Gamma_{\frac{n-m}{2}}(-t-\frac{m}{2})} \nonumber \\
& =\prod_{j=0}^{\frac{n-m-4}{4}}\frac{\Gamma(\tfrac{1}{2}-j)}{\Gamma(-[\tfrac{m}{2}]-\tfrac{1}{2}-j)}
        \cdot
\left(\lim_{t\to 0}\frac{\Gamma(t-j)}{\Gamma(-t-[\frac{m+1}{2}]-j)}\right)
        \nonumber \\
& =(-1)^{[\frac{m+3}{2}]}\prod_{j=0}^{\frac{n-m-4}{4}}\frac{\Gamma(\tfrac{1}{2}-j)}
     {\Gamma(-[\tfrac{m}{2}]-\tfrac{1}{2}-j)}
    \cdot \frac{([\frac{m+1}{2}]+j)!}{j!}.
 \label{gammanmII}
\end{align}
In the calculation above, we used identity
\begin{equation}
\label{Id}
\lim_{t\to 0}\frac{\Gamma(t-j)}{\Gamma(-t-a-j)}=(-1)^{a+1}\frac{(a+j)!}{j!},
\end{equation}
for $a\in\mathbb{N}$.
Combining (\ref{gammanmI}), (\ref{gammanmII}), and (\ref{gamma}), we obtain (\ref{gammanm}).
This completes the proof of Proposition \ref{P1}.
\end{proof}
\subsubsection{Analysis of $\xi$-factor $\boldsymbol{\xi_{n,m}(s)}$}
\label{xi}
\begin{Prop}
\label{P2}
{\it 
Function $\xi_{n,m}(s)$ in {\rm (\ref{Fnm})} has a simple pole at $s=\frac{m+1}{2}$ and
\begin{equation}
\label{propxinm}
\mathop{{\rm Res}}_{s=\frac{m+1}{2}}\xi_{n,m}(s)
=
\tfrac{1}{4}
\cdot
\frac{\xi\left(\tfrac{n-m}{2}\right)}{\xi\left(\tfrac{n+m+2}{2}\right)}\,
        \frac{\prod_{\substack{1\leq j\leq [\frac{n}{2}] \\ j\ne\frac{n-m}{2}}}\xi(n-m-2j)}
             {\prod_{j=1}^{[\frac{n}{2}]}\xi(n+m+2-2j)}.
\end{equation}
}
\end{Prop}
To prove Proposition \ref{P2},
we decompose $\xi_{n,m}(s)$ into two factors as follows:
\begin{align}
\xi_{n,m}(s)&= \xi_{n,m}^{(\text{I})}(s)\,\xi_{n,m}^{(\text{II})}(s), \label{xidecomp}\\
   &
    \xi_{n,m}^{(\text{I})}(s)=\frac{\xi\left(\frac{n+m+2}{2}-2s\right)}{\xi\left(2s+\frac{n-m}{2}\right)}
    \frac{\prod_{\substack{1\leq j\leq [\frac{n}{2}] \\ j\ne\frac{n-m}{2}}}\xi((n+m+2)-4s-2j)}
    {\prod_{j=1}^{[\frac{n}{2}]}\xi(4s+(n-m)-2j)}, \nonumber \\
   & \xi_{n,m}^{(\text{II})}(s)=\xi(2m+2-4s).\nonumber
\end{align}
This decomposition means that factor $\xi_{n,m}^{(\text{I})}(s)$ is obtained by extracting
factor $\xi_{n,m}^{(\text{II})}(s)=\xi(2m+2-4s)$ from $\xi_{n,m}(s)$.
\begin{Lem}
\label{xinm}
\;{\it 
{\rm (1)} Function $\xi_{n,m}^{({\rm I})}(s)$ is holomorphic at $s=\frac{m+1}{2}$
 and
\begin{equation}
\label{xiI}
\xi_{n,m}^{({\rm I})}\left(\tfrac{m+1}{2}\right)
=
\frac{\xi\left(\tfrac{n-m}{2}\right)}{\xi\left(\tfrac{n+m+2}{2}\right)}\,
        \frac{\prod_{\substack{1\leq j\leq [\frac{n}{2}] \\ j\ne\frac{n-m}{2}}}\xi(n-m-2j)}
             {\prod_{j=1}^{[\frac{n}{2}]}\xi(n+m+2-2j)}.
\end{equation}
{\rm (2)}\; Function $\xi_{n,m}^{({\rm II})}(s)$ has a simple pole at $s=\frac{m+1}{2}$ and
\begin{equation}
\label{xiII}
\mathop{{\rm Res}}_{s=\frac{m+1}{2}}
\xi_{n,m}^{({\rm II})}(s)=\tfrac{1}{4}.
\end{equation}
}
\end{Lem}
\begin{proof}
(1) Functions
$$
\tfrac{n+m+2}{2}-2s,\quad 2s+\tfrac{n-m}{2},\quad 4s+(n-m)-2j\;\;
\left(1\leq j\leq \left[\tfrac{n}{2}\right]\right)
$$
have integral values with $\geq 2$ at $s=\frac{m+1}{2}$. Hence, functions
$$
\frac{\xi\left(\frac{n+m+2}{2}-2s\right)}{\xi\left(2s+\frac{n-m}{2}\right)}\quad
\text{and}\quad
\prod_{j=1}^{[\frac{n}{2}]}\xi(4s+(n-m)-2j)
$$
are holomorphic at $s=\frac{m+1}{2}$. We consider factors
$$
\xi((n+m+2)-4s-2j)\quad \left( 1\leq j\leq\left[\tfrac{n}{2}\right],\;j\ne \tfrac{n-m}{2}\right).
$$
For $j$ with $1\leq j<\frac{n-m}{2}$, function $(n+m+2)-4s-2j$ has positive
even values with $\geq 2$ at $s=\frac{m+1}{2}$. Hence, $\xi((n+m+2)-4s-2j)$ is
holomorphic at $s=\frac{m+1}{2}$.\\
For $j$ with $\frac{n-m}{2}< j\leq\left[\frac{n}{2}\right]$, function
$(n+m+2)-4s-2j$ has a negative even value with $\leq -2$ at $s=\frac{m+1}{2}$.
Hence, $\xi((n+m+2)-4s-2j)$ is also holomorphic at $s=\frac{m+1}{2}$. 
Consequently, we have expression $\xi_{n,m}^{({\rm I})}\left(\tfrac{m+1}{2}\right)$
in (\ref{xiI}).
\\
(2)\; Because $\xi(s)$ has a simple pole at $s=0$, $\xi_{n,m}^{(\text{II})}(s)=\xi(2m+2-4s)$
has a simple pole at $s=\frac{m+1}{2}$, and the residue is
\begin{align*}
\mathop{{\rm Res}}_{s=\frac{m+1}{2}}
\xi_{n,m}^{(\text{II})}(s) &=
\mathop{{\rm Res}}_{s=\frac{m+1}{2}}\xi (2m+2-4s)\\
& =\mathop{{\rm Res}}_{s=\frac{m+1}{2}}\Gamma(m+1-2s)\cdot \zeta (0)=\tfrac{1}{4}.
\end{align*}
This proves (2).
\end{proof}
From this lemma, it is clear that $\xi_{n,m}(s)=\xi_{n,m}^{(\text{I})}(s)\cdot \xi_{n,m}^{(\text{II})}(s)$
has a simple pole at $s=\frac{m+1}{2}$ and
\begin{align*}
\mathop{{\rm Res}}_{s=\frac{m+1}{2}}
\xi_{n,m}(s) &=\xi_{n,m}^{(\text{I})}\left(\tfrac{m+1}{2}\right)
\mathop{{\rm Res}}_{s=\frac{m+1}{2}}
\xi_{n,m}^{(\text{II})}(s) \\
& =
\tfrac{1}{4}\cdot
\frac{\xi\left(\tfrac{n-m}{2}\right)}{\xi\left(\tfrac{n+m+2}{2}\right)}\,
        \frac{\prod_{\substack{1\leq j\leq [\frac{n}{2}] \\ j\ne\frac{n-m}{2}}}\xi(n-m-2j)}
             {\prod_{j=1}^{[\frac{n}{2}]}\xi(n+m+2-2j)}.
\end{align*}
This completes the proof of Proposition \ref{P2}.
\vspace{3mm}
\\
We return to the proof of Theorem \ref{Residue1}. We recall expression
\begin{align*}
E_{\frac{n-m}{2}}^{(n)}(Z,s)=& F_{n,m}(s)\cdot E_{\frac{n-m}{2}}^{(n)}\left(Z,\tfrac{m+1}{2}-s\right),\\
                                      & F_{n,m}(s)=\gamma_{n,m}(s)\cdot\xi_{n,m}(s).
\end{align*}
Because $\gamma_{n,m}(s)$ is holomorphic at $s=\frac{m+1}{2}$ (Proposition \ref{P1})
and $\xi_{n,m}(s)$ has a simple pole at  $s=\frac{m+1}{2}$ (Proposition \ref{P2}),
we obtain
\begin{align*}
\mathop{{\rm Res}}_{s=\frac{m+1}{2}}
E_{\frac{n-m}{2}}^{(n)}(Z,s) &=
\mathop{{\rm Res}}_{s=\frac{m+1}{2}}F_{n,m}(s)\cdot E_{\frac{n-m}{2}}^{(n)}(Z,0)\\
  &= \gamma_{n,m}\left(\tfrac{m+1}{2}\right)
   \mathop{{\rm Res}}_{s=\frac{m+1}{2}}\xi_{n,m}(s)\cdot E_{\frac{n-m}{2}}^{(n)}(Z,0).
\end{align*}
Therefore, if we set
\begin{align*}
c_{n,m}&:=\gamma_{n,m}\left(\tfrac{m+1}{2}\right)\cdot
   \mathop{{\rm Res}}_{s=\frac{m+1}{2}}\xi_{n,m}(s)\\
   &=(-1)^{[\frac{m+3}{2}]}\cdot\tfrac{1}{4}\cdot
      \frac{\Gamma_{\frac{n-m}{2}}\left(\tfrac{n-m}{2}\right)}{\Gamma_{\frac{n-m}{2}}\left(\tfrac{n+1}{2}\right)}
\prod_{j=0}^{\frac{n-m-4}{4}}\frac{\Gamma(\tfrac{1}{2}-j)}{\Gamma(-[\tfrac{m}{2}]-\tfrac{1}{2}-j)}
    \cdot \prod_{l=1}^{[\frac{m+1}{2}]}(j+l)
 \\
 & \cdot \frac{\xi\left(\tfrac{n-m}{2}\right)}{\xi\left(\tfrac{n+m+2}{2}\right)}\,
        \frac{\prod_{\substack{1\leq j\leq [\frac{n}{2}] \\ j\ne\frac{n-m}{2}}}\xi(n-m-2j)}
             {\prod_{j=1}^{[\frac{n}{2}]}\xi(n+m+2-2j)},
\end{align*}
then
$$
\label{Res}
\mathop{{\rm Res}}_{s=\frac{m+1}{2}}E_{\frac{n-m}{2}}^{(n)}(Z,s)=c_{n,m}\,E_{\frac{n-m}{2}}^{(n)}(Z,0),
$$
and this proves Theorem \ref{Residue1}.
\begin{Rem}
It is noteworthy that
Weissauer proved the statement in our Theorem \ref{Residue1} using a different method
(\cite[p.131, (175)]{Weis}). That is, he showed that 
$\mathop{{\rm Res}}_{s=\frac{m+1}{2}}E_{\frac{n-m}{2}}^{(n)}(Z,s)$ is a constant multiple
of $E_{\frac{n-m}{2}}^{(n)}(Z,0)$.
\end{Rem}

\subsection{Refinement of Shimura's result}
{\rm
We prove the following theorem, which is a refinement of Shimura's result
(Theorem \ref{ShimuraMain}).}
\begin{Thm}
\label{Residue2}{\it 
Assume that $\frac{n-1}{2} \equiv 0 \pmod{2}$. Then, we have
$$
\mathop{{\rm Res}}_{s=1}E_{\frac{n-1}{2}}^{(n)}(Z,s)=\pi^{-n}\cdot c_{n}\,E_{\frac{n-1}{2}}^{(n)}(Z,0),
$$
where constant $c_n$ is expressed as
\begin{equation}
\label{cn}
c_n=
 (-1)^{\frac{n-1}{4}}2^{-2n-2}\cdot 3^{-1}
    \frac{\left(\tfrac{n+3}{2}\right)!(n+1)!}
           {\left(\tfrac{n-1}{2}\right)!} 
     \frac{B_{\frac{n-1}{2}}} {B_{\frac{n+3}{2}}B_{n+1}B_{n-1}}.
\end{equation}
Here, $B_m$ is the $m$-th Bernoulli number.
}
\end{Thm}
\begin{Rem}
The theorem above is considered a special case ($m=1$) of Theorem \ref{Residue1}.
The theorem asserts that constant $c_{n,1}$ in Theorem \ref{Residue1} can be expressed as
$c_{n,1}=\pi^{-n}\cdot c_n$ with $c_n\in\mathbb{Q}^{\times}$.
\end{Rem}
We use the notation in the previous section as $m=1$.\\
We recall the functional equation (\ref{FEQ}), 
$$
E_{\frac{n-1}{2}}^{(n)}(Z,s)=F_{n,1}(s)\cdot E_{\frac{n-1}{2}}^{(n)}(Z,1-s).
$$
Because $E_{\frac{n-1}{2}}^{(n)}(Z,s)$ is holomorphic at $s=0$, it is sufficient to prove
that
$$
\mathop{{\rm Res}}_{s=1}F_{n,1}(s)=\pi^{-n}\cdot c_n.
$$
We recall the definition of $F_{n,1}(s)$:
\begin{align*}
F_{n,1}(s) &= \gamma_{n,1}(s)\cdot\xi_{n,1}(s),\\
   & \gamma_{n,1}(s)=
   \frac{\Gamma_n\left(\frac{n+1}{2}-s\right)}{\Gamma_n\left( \frac{n+3}{4}-s\right)}\,
   \frac{\Gamma_n\left(s+\frac{n-1}{4}\right)}{\Gamma_n\left( s+\frac{n-1}{2} \right)},\\
   & \xi_{n,1}(s)=
   \frac{\xi\left(\frac{n+3}{2}-2s \right)}{\xi\left(2s+\frac{n-1}{2}\right)}
   \prod_{j=1}^{\frac{n-1}{2}}
   \frac{\xi(n+3-4s-2j)}{\xi(4s+n-1-2j)}.
\end{align*}
%
\subsubsection{Analysis of $\boldsymbol{\gamma}$-part}
\begin{Prop}
\label{PP1}
$$
\lim_{s\to 1}\gamma_{n,1}(s)=
(-1)^{\frac{n-1}{4}}2^{\frac{n-1}{4}}
                               \frac{\left(\tfrac{n-1}{2}\right)!\left(\tfrac{n-1}{4}\right)!}{(n-1)!}
                               \left(\tfrac{n-3}{2}\right)!!\,.
$$
\end{Prop}
\begin{proof}
Similar to (\ref{gamma}), we express $\gamma_{n,1}(s)$ as
\begin{align*}
\gamma_{n,1}(s) =& \gamma_{n,1}^{(\text{I})}(s)\gamma_{n,1}^{(\text{II})}(s)\\
 & \gamma_{n,1}^{(\text{I})}(s)=\frac{\Gamma_{\frac{n-1}{2}}\left( \frac{n+1}{2}-s\right)}
                                           {\Gamma_{\frac{n-1}{2}}\left( s+\frac{n-1}{2}\right)},\quad
\gamma_{n,1}^{(\text{II})}(s)=\frac{\Gamma_{\frac{n-1}{2}}\left( s-\frac{1}{2}\right)}
                                           {\Gamma_{\frac{n-1}{2}}\left( \frac{1}{2}-s\right)}.
\end{align*}
\textbf{Calculation of} $\boldsymbol{\gamma_{n,1}^{(\text{I})}(s)}$:
\vspace{1mm}
\\
Because $\gamma_{n,1}^{(\text{I})}(s)$ is holomorphic at $s=1$, 
from (\ref{gammanmI}), we have
\begin{align}
\gamma_{n,1}^{(\text{I})}(1) &= \frac{\Gamma_{\frac{n-1}{2}}\left(\frac{n-1}{2}\right)}
                                              {\Gamma_{\frac{n-1}{2}}\left(\frac{n+1}{2}\right)}
                                    =\prod_{j=0}^{\frac{n-5}{4}}\frac{\Gamma\left(\frac{n-1}{2}-j \right)}
                                            {\Gamma\left(\frac{n+1}{2}-j\right)}\cdot
                                     \prod_{j=0}^{\frac{n-5}{4}}\frac{\Gamma\left(\frac{n-2}{2}-j \right)}
                                            {\Gamma\left(\frac{n}{2}-j\right)}   \nonumber \\
                                  &=\frac{\Gamma\left(\frac{n+3}{4} \right)}{\Gamma\left(\frac{n+1}{2 }\right)}
                                    \cdot \frac{\Gamma\left(\frac{n+1}{4} \right)}{\Gamma\left(\frac{n}{2}\right)}
                                    =2^{\frac{n-1}{2}}\frac{\Gamma\left(\frac{n+1}{2}\right)}{\Gamma(n)}
      = 2^{\frac{n-1}{2}}\frac{\left(\tfrac{n-1}{2}\right)!}{(n-1)!}. \label{gamman1I}
\end{align}
\textbf{Calculation of} $\boldsymbol{\gamma_{n,1}^{(\text{II})}(s)}$:
\vspace{1mm}
\\
We set $m=1$ in (\ref{gammanmII}). Subsequently, we have
\begin{align*}
\lim_{s\to 1}\gamma_{n,1}^{(\text{II})}(s) & =
\prod_{j=0}^{\frac{n-5}{4}}\frac{\Gamma(\tfrac{1}{2}-j)}{\Gamma(-\tfrac{1}{2}-j)}
        \cdot
\left(\lim_{t\to 0}\frac{\Gamma(t-j)}{\Gamma(-t-1-j)}\right) \\
& =
\frac{\Gamma(\tfrac{1}{2})}{\Gamma\left(-\tfrac{n-3}{4}\right)}
\cdot \left(\tfrac{n-1}{4}\right)!.
\end{align*}
Simple calculation shows
$$                                           
                                            \frac{\Gamma(\tfrac{1}{2})}{\Gamma\left(-\tfrac{n-3}{4}\right)}
                                            =\prod_{j=0}^{\frac{n-5}{4}}\left(-\tfrac{1}{2}-j \right)
                                            =(-1)^{\frac{n-1}{4}}\,2^{-\frac{n-1}{4}}\left(\tfrac{n-3}{2}\right)!!.
$$                        
Therefore, we obtain
\begin{equation}
\label{gamman1II}
\lim_{s\to 1}\gamma_{n,1}^{(\text{II})}(s)=(-1)^{\frac{n-1}{4}}2^{-\frac{n-1}{4}}\left(\tfrac{n-1}{4}\right)!
                                                \left(\tfrac{n-3}{2}  \right)!! .
\end{equation}
Combining (\ref{gamman1I}) and (\ref{gamman1II}), we obtain
\begin{align}
\label{gammafinal}
\lim_{s\to 1}\gamma_{n,1}(s) &=\gamma_{n,1}^{(\text{I})}(1)\cdot \lim_{s\to 1}\gamma_{n,1}^{(\text{II})}(s)
                                              \nonumber \\
                               &=(-1)^{\frac{n-1}{4}}2^{\frac{n-1}{4}}
                               \frac{\left(\tfrac{n-1}{2}\right)!\left(\tfrac{n-1}{4}\right)!}{(n-1)!}
                               \left(\tfrac{n-3}{2}\right)!!\,.
\end{align}
\end{proof}
\subsubsection{Analysis of $\boldsymbol{\xi}$-part}
\begin{Prop}\label{PP2}
\begin{align*}
& \mathop{{\rm Res}}_{s=1}\,\xi_{n,1}(s)\\
& \qquad =
\pi^{-n}\cdot 2^{-3-2n}\cdot 3^{-1}
\frac{\left(\tfrac{n-5}{4}\right)!\left(\tfrac{n+3}{2}\right)!(n+1)!(n-1)!}
        {\left(\tfrac{n-1}{4}\right)!\left(\tfrac{n-3}{2}\right)!
        \left\{\left(\tfrac{n-1}{2}\right)!\right\}^2}
\frac{B_{\frac{n-1}{2}}}{B_{\frac{n+3}{2}}B_{n+1}B_{n-1}}. 
\end{align*}
\end{Prop}
\begin{proof}
We recall the decomposition of $\xi_{n,m}(s)$ in (\ref{xidecomp}).
If we apply this decomposition to $\xi_{n,1}(s)$, we obtain the following expression:
\begin{align*}
\xi_{n,1}(s) =& \xi_{n,1}^{(\text{I})}(s)\cdot\xi_{n,1}^{(\text{II})}(s)\\
      & \xi_{n,1}^{(\text{I})}(s)=
      \frac{\xi\left(\frac{n+3}{2}-2s \right)}{\xi\left(2s+\frac{n-1}{2} \right)}
\frac{\prod_{j=1}^{\frac{n-3}{2}}\xi(n+3-4s-2j)}{\prod_{j=1}^{\frac{n-1}{2}}\xi(4s+n-1-2j)},\\
&  \xi_{n,1}^{(\text{II})}(s)=\xi (4-4s).
\end{align*}
First, we calculate the value of $\xi_{n,1}^{(\text{I})}(1)$. (The holomorphy
of $\xi_{n,1}^{(\text{I})}(s)$ at $s=1$ is guaranteed by Lemma \ref{xinm}, (1).)
\begin{align}
\label{rho}
\xi_{n,1}^{(\text{I})}(1)  & = 
\frac{\xi\left(\frac{n-1}{2}\right)}{\xi\left(\frac{n+3}{2} \right)}\cdot
\frac{\prod_{j=1}^{\frac{n-3}{2}}\xi(n-1-2j)}{\prod_{j=1}^{\frac{n-1}{2}}\xi(n+3-2j)} \nonumber \\
   &= \frac{\xi\left(\frac{n-1}{2}\right)}{\xi\left(\frac{n+3}{2} \right)}\cdot
         \frac{\xi (2)}{\xi (n+1)\,\xi(n-1)}.
\end{align}
We rewrite the factors appearing in the last formula using the Bernoulli numbers as follows:
$$
\frac{\xi\left(\frac{n-1}{2}\right)}{\xi\left(\frac{n+3}{2} \right)}
=
-2^{-2}\pi^{-1}\frac{\left(\tfrac{n-5}{4}\right)!\left(\tfrac{n+3}{2}\right)!}
                             {\left(\tfrac{n-1}{4}\right)!\left(\tfrac{n-1}{2}\right)!}
                    \, \frac{B_{\frac{n-1}{2}}}{B_{\frac{n+3}{2}}}
$$
and
$$
 \frac{\xi (2)}{\xi (n+1)\,\xi(n-1)}=
-\pi^{1-n}\cdot 2^{1-2n}\cdot 3^{-1}
\frac{(n+1)!(n-1)!}{\left(\tfrac{n-1}{2}\right)!\left(\tfrac{n-3}{2}\right)!}\,
\frac{1}{B_{n+1}B_{n-1}}.
$$
Consequently,
\begin{align*}
& \xi_{n,1}^{(\text{I})}(1)  \nonumber \\
& =
\pi^{-n}\cdot 2^{-1-2n}\cdot 3^{-1}
\frac{\left(\tfrac{n-5}{4}\right)!\left(\tfrac{n+3}{2}\right)!(n+1)!(n-1)!}
        {\left(\tfrac{n-1}{4}\right)!\left(\tfrac{n-3}{2}\right)!
        \left\{\left(\tfrac{n-1}{2}\right)!\right\}^2}
\frac{B_{\frac{n-1}{2}}}{B_{\frac{n+3}{2}}B_{n+1}B_{n-1}}.
\end{align*}
Next, we consider factor $\xi_{n,1}^{(\text{II})}(s)=\xi(4-4s)$.
This function has a simple pole at $s=1$, and the residue is provided
in (\ref{xiII}). That is,
$$
\mathop{{\rm Res}}_{s=1}\xi_{n,1}^{(\text{II})}(s)=\mathop{{\rm Res}}_{s=1}\xi(4-4s)=\tfrac{1}{4}.
$$
Hence, we obtain
\begin{align}
\label{xifinal}
&\mathop{{\rm Res}}_{s=1}\,\xi_{n,1}(s) = \xi_{n,1}^{(\text{I})}(1) \cdot 
                                                     \mathop{{\rm Res}}_{s=1}\xi_{n,1}^{(\text{II})}(s)
               \nonumber \\
&=\pi^{-n}\cdot 2^{-3-2n}\cdot 3^{-1}
\frac{\left(\tfrac{n-5}{4}\right)!\left(\tfrac{n+3}{2}\right)!(n+1)!(n-1)!}
        {\left(\tfrac{n-1}{4}\right)!\left(\tfrac{n-3}{2}\right)!
        \left\{\left(\tfrac{n-1}{2}\right)!\right\}^2}
\frac{B_{\frac{n-1}{2}}}{B_{\frac{n+3}{2}}B_{n+1}B_{n-1}}.                                   
\end{align}
This completes the proof of Proposition \ref{PP2}.
\end{proof}
Summarizing (\ref{gammafinal}) and (\ref{xifinal}), we conclude
that
\begin{align*}
& \mathop{{\rm Res}}_{s=1}\,F_{n,1}(s)=
\lim_{s\to 1}\gamma_{n,1}(s)\cdot \mathop{{\rm Res}}_{s=1}\,\xi_{n,1}(s)\\
& \qquad =\pi^{-n}\cdot
     (-1)^{\frac{n-1}{4}}2^{-\frac{7n+17}{4}}\cdot 3^{-1}
    \frac{\left(\tfrac{n+3}{2}\right)!(n+1)!}
           {\left(\tfrac{n-1}{2}\right)!} 
     \frac{B_{\frac{n-1}{2}}} {B_{\frac{n+3}{2}}B_{n+1}B_{n-1}}.
\end{align*}
We used identities
$$
\frac{\left(\frac{n-3}{2}\right)!!}{\left(\frac{n-3}{2}\right)!}=\frac{1}{\left(\frac{n-5}{2}\right)!!},\quad
\frac{\left(\frac{n-5}{4}\right)!}{\left(\frac{n-5}{2}\right)!!}=2^{\frac{5-n}{4}}
$$ 
in the above calculation (note that $n \equiv 1 \pmod{4}$). 
This completes the proof of Theorem \ref{Residue2}. We obtained
$$
\mathop{{\rm Res}}_{s=1}E_{\frac{n-1}{2}}^{(n)}(Z,s)=\pi^{-n}\cdot c_{n}\,E_{\frac{n-1}{2}}^{(n)}(Z,0)
$$
with
$$
c_n=
 (-1)^{\frac{n-1}{4}}2^{-2n-2}\cdot 3^{-1}
    \frac{\left(\tfrac{n+3}{2}\right)!(n+1)!}
           {\left(\tfrac{n-1}{2}\right)!} 
     \frac{B_{\frac{n-1}{2}}} {B_{\frac{n+3}{2}}B_{n+1}B_{n-1}}.
$$
This completes the proof of Theorem \ref{Residue2}.
\section{Hermtian Eisenstein series}
\label{hermitian}
In this section, we treat a case of the Hermitian Eisenstein series
$E_{k,\boldsymbol{K}}^{(n)}(Z,s)$ (for the precise definition,
see $\S$ \ref{HerEis}), and provide results analogous to those of the Siegel Eisenstein case.
The objective is to refine Shimura's result (Theorem \ref{ShimuraMain}) for
Hermitian Eisenstein series.
\subsection{Hermitian modular forms}
Let $\mathcal{H}_n$ be the Hermitian upper half space of degree $n$
defined by
$$
\mathcal{H}_n=\{\,Z\in M_n(\mathbb{C})\,|\,I(Z):=\tfrac{1}{2\sqrt{-1}}(Z-{}^t\overline{Z})>0\,\}.
$$
The special unitary group $SU(n,n)$ is realized by
$$
G_n:=\{\,M\in SL_{2n}(\mathbb{C})\,|\,{}^t\overline{M}J_nM=J_n\,\},
$$
where $J_n={\scriptsize \begin{pmatrix}0_n & E_n \\ -E_n & 0_n \end{pmatrix}}$.
The group $G_n$ acts on $\mathcal{H}_n$ by generalized
linear fractional transformations.

Let $\boldsymbol{K}$ be an imaginary quadratic number field with discriminant
$-D_{\boldsymbol{K}}$. We denote by $\mathcal{O}_{\boldsymbol{K}}$ and
$\mathfrak{d}_{\boldsymbol{K}}$ the ring of integers in $\boldsymbol{K}$ and the
different ideal of $\boldsymbol{K}$, respectively. Let $\chi_{\boldsymbol{K}}$ be the
Kronecker character of $\boldsymbol{K}$ and $h_{\boldsymbol{K}}$ the
class number of $\boldsymbol{K}$. We define the {\it Hermitian modular
group} of degree $n$ for $\boldsymbol{K}$ by
$$
\varGamma_{n,\boldsymbol{K}}=G_n\cap M_{2n}(\mathcal{O}_{\boldsymbol{K}}).
$$
We denote by $M_k(\varGamma_{n,\boldsymbol{K}})$ the $\mathbb{C}$-vector
space of Hermitian modular forms of weight $k$ for $\varGamma_{n,\boldsymbol{K}}$.

It is known that each $F\in M_k(\varGamma_{n,\boldsymbol{K}})$ admits a
Fourier expansion of the form
$$
F(Z)=\sum_{0\leq H\in \Lambda_n(\boldsymbol{K})}a_F(H)
\text{exp}(2\pi\sqrt{-1}\text{tr}(HZ)),
$$
where 
$$
\Lambda_n(\boldsymbol{K}):=\{\, H=(h_{ij})\in \text{Her}(\boldsymbol{K})\,|\,
h_{ii}\in\mathbb{Z},\,h_{ij}\in\mathfrak{d}_{\boldsymbol{K}}^{-1}\,\}.
$$
As in the Siegel modular case, we define $M_k(\varGamma_{n,\boldsymbol{K}})_R$
for a subring, $R\subset\mathbb{C}$.
\subsubsection{Hermitian Eisenstein series}
\label{HerEis}
We define a parabolic subgroup of $G_n$ as follows:
$$
P_n:=\left\{\,{\scriptsize \begin{pmatrix} A & B \\ C & D \end{pmatrix}}\in G_n\,\Big{|}\;\;
                 C=0_n\,\right\}.
$$
The Eisenstein series considered in this section is
$$
E_{k,\boldsymbol{K}}^{(n)}(Z,s):=
\text{det}(I(Z))^s\!\!\sum_{\binom{*\,*}{C\,D}\in(P_n\cap \Gamma_{n,\boldsymbol{K}})\backslash
\Gamma_{n,\boldsymbol{K}}}
\!\!\text{det}(CZ+D)^{-k}|\text{det}(CZ+D)|^{-2s},
$$
where $(Z,s)\in \mathcal{H}_n\times\mathbb{C}$, $k\in 2\mathbb{Z}_{\geq 0}$.
It is known that this series is absolutely, uniformly convergent if $\text{Re}(s)+k>2n$.
Therefore, $E_{k,\boldsymbol{K}}^{(n)}(Z):=E_{k,\boldsymbol{K}}^{(n)}(Z,0)$ becomes
an element of $M_k(\varGamma_{n,\boldsymbol{K}})$ if $k>2n$. Moreover, it has rational
Fourier coefficients (i.e., $E_{k,\boldsymbol{K}}^{(n)}(Z)\in M_k(\varGamma_{n,\boldsymbol{K}})_{\mathbb{Q}}$).

We refer to $E_{k,\boldsymbol{K}}^{(n)}(Z,s)$ as the {\it Hermitian Eisenstein series}
of degree $n$. \\
Next, we study the analytic property of the Hermitian Eisenstein series.
\subsubsection{Functional equation of Hermitian Eisenstein series}
\label{FCHermitian}
In the remainder of this section ($\S$ \ref{hermitian}), we apply the assumption that
$$
h_{\boldsymbol{K}}=1.
$$
For $n\in \mathbb{Z}_{>0}$ and $k\in 2\mathbb{Z}_{\geq 0}$, we define
function $g_{n,k,\boldsymbol{K}}(s)$ by
\begin{equation}
\label{defg}
g_{n,k,\boldsymbol{K}}(s)
=\frac{\Gamma_{n,\mathbb{C}}(\frac{s+k}{2})}{\Gamma_{n,\mathbb{C}}(\frac{s}{2})}
\prod_{j=0}^{n-1}\xi(s-j;\chi_{\boldsymbol{K}}^j),
\end{equation}
where
\begin{align*}
& \Gamma_{n,\mathbb{C}}(s)=\pi^{\frac{n(n-1)}{2}}\prod_{j=0}^{n-1}\Gamma(s-j),\\
& \xi(s;\chi_{\boldsymbol{K}}^j)
    =\begin{cases}
    \pi^{-\frac{s}{2}}\Gamma(\tfrac{s}{2})\zeta(s) & \text{if $j$\, is\, even},
    \vspace{2mm}
    \\
    D_{\boldsymbol{K}}^{\frac{s}{2}}\,\pi^{-\frac{s}{2}}\Gamma(\tfrac{s+1}{2})L(s;\chi_{\boldsymbol{K}})
    & \text{if $j$\, is\, odd},
    \end{cases}
\end{align*}
and $L(s;\chi)$ is the Dirichlet $L$-function.
Under the aforementioned assumption, the Hermitian Eisenstein has the functional equation
of the form
\begin{equation}
\label{FCH}
E_{k,\boldsymbol{K}}^{(n)}(Z,s)
=\frac{g_{n,k,\boldsymbol{K}}(2n-k-2s)}{g_{n,k,\boldsymbol{K}}(2s+k)}
E_{k,\boldsymbol{K}}^{(n)}(Z,n-k-s)
\end{equation}
(e.g., cf. \cite{N}, Theorem 1.8).
\subsubsection{Shimura's result in Hermitian case}
Shimura's result for the Hermitian Eisenstein series is as
follows:
\begin{Thm} (Shimura \cite{Sh}, {\scshape Proposition 10.3})
\label{ShimuraMainH}{\it 
The Eisenstein series\\
 $E_{n-1,\boldsymbol{K}}^{(n)}(Z,s)$ has at most a simple pole
at $s=1$. The residue is written as
$$
\mathop{{\rm Res}}_{s=1}E_{n-1,\boldsymbol{K}}^{(n)}(Z,s)=\pi^{-n}\cdot f
$$
with some $f\in M_{n-1}(\varGamma_{n,\boldsymbol{K}})_{\mathbb{Q}}$.}
\end{Thm}
\subsection{Main result for Hermitian Eisenstein series}
Using the functional equation (\ref{FCH}), we can refine Shimura's result
for the Hermitian Eisenstein series.
\begin{Thm} \textbf{(SU case)}
\label{SU1}
{\it We assume that $h_{\boldsymbol{K}}=1$.
Let $n$ and $m$ be integers satisfying $n>m\geq 1$ and 
$n-m\in 2\mathbb{Z}_{>0}$.
Then, the Eisenstein series $E_{n-m,\boldsymbol{K}}^{(n)}(Z,s)$ has at most a simple
pole at $s=m$, and
\begin{equation}
\label{ResH}
\mathop{{\rm Res}}_{s=m}E_{n-m,\boldsymbol{K}}^{(n)}(Z,s)
=c_{n,m,\boldsymbol{K}}\cdot E_{n-m,\boldsymbol{K}}^{(n)}(Z,0).
\end{equation}
Here,
\begin{align*}
c_{n,m,\boldsymbol{K}} &=
                              (-1)^{m+1}\,\tfrac{1}{2}\cdot \prod_{j=0}^{n-m-2}\frac{(m+j)!}{j !}
                              \frac{\Gamma_{\frac{n-m}{2},\mathbb{C}}(n-m)}
                                     {\Gamma_{\frac{n-m}{2},\mathbb{C}}(n)}\\
                            & \quad \cdot
                            \frac{\prod_{\substack{0\leq j\leq n-1 \\ j\ne n-m}}\xi(n-m-j;\chi_{\boldsymbol{K}}^j)}
                            {\prod_{j=0}^{n-1}\xi(n+m-j;\chi_{\boldsymbol{K}}^j)} .
\end{align*}
}
\end{Thm}
\begin{Cor}{\it
If $n-m \equiv 2 \pmod{4}$, $E_{n-m,\boldsymbol{K}}^{(n)}(Z,s)$ is holomorphic at
$s=m$.}
\end{Cor}
The corollary arises from the theory of singular modular forms in the Hermitian
modular case.
\vspace{4mm}
\\
{\it Proof of Theorem} \ref{SU1}.\;
Considering (\ref{FCH}), we set
$$
G_{n,m,\boldsymbol{K}}(s)
:=\left.\frac{g_{n,k,\boldsymbol{K}}(2n-k-2s)}{g_{n,k,\boldsymbol{K}}(2s+k)}\right|_{k=n-m}
=\frac{g_{n,n-m,\boldsymbol{K}}(n+m-2s)}{g_{n,n-m,\boldsymbol{K}}(2s+n-m)}.
$$
If we use this notation, the functional equation (\ref{FCH}) in this case can be written as
$$
E_{n-m,\boldsymbol{K}}^{(n)}(Z,s)=G_{n,m,\boldsymbol{K}}(s)
                                      E_{n-m,\boldsymbol{K}}^{(n)}(Z,m-s).
$$
To prove Theorem \ref{SU1}, it is sufficient to show the following:
\vspace{2mm}
\\
(i)\quad 
$E_{n-m,\boldsymbol{K}}^{(n)}(Z,s)$ is holomorphic at $s=0$.
\vspace{1mm}
\\
(ii)\quad
$G_{n,m,\boldsymbol{K}}(s)$ has a simple pole at $s=m$, and the residue is
$c_{n,m,\boldsymbol{K}}$.
\vspace{4mm}
\\
If these two statements are proven, we then obtain
\begin{equation*}
\mathop{{\rm Res}}_{s=m}E_{n-m,\boldsymbol{K}}^{(n)}(Z,s)
=\mathop{{\rm Res}}_{s=m}G_{n,m,\boldsymbol{K}}(s)\cdot
 E_{n-m,\boldsymbol{K}}^{(n)}(Z,0).
\end{equation*}
Next, we prove (i) and (ii).
\\
Statement (i) is based on Shimura's results in \cite{Sh}.
\\
We prove (ii) according to the discussion regarding the
Siegel Eisenstein series.
\\
We consider the following expression of $G_{n,m,\boldsymbol{K}}(s)$, as in (\ref{Fnm}):
\begin{align}
\label{DecH}
G_{n,m,\boldsymbol{K}}(s) &=\gamma_{n,m,\boldsymbol{K}}(s)\,\xi_{n,m,\boldsymbol{K}}(s),\\
   & \gamma_{n,m,\boldsymbol{K}}(s):=\frac{\Gamma_{n,\mathbb{C}}(n-s)}
                                                              {\Gamma_{n,\mathbb{C}}(s+\frac{n+m}{2})}
                                                   \cdot \frac{\Gamma_{n,\mathbb{C}}(s+\frac{n-m}{2})}
                                                              {\Gamma_{n,\mathbb{C}}(s+n-m)}
                                                              \nonumber
                                                              \\
      & \xi_{n,m,\boldsymbol{K}}(s):=\prod_{j=0}^{n-1}
                                                \frac{\xi(n+m-2s-j;\chi_{\boldsymbol{K}}^j)}
                                                {\xi(2s+n-m-j;\chi_{\boldsymbol{K}}^j)}.
                                                \nonumber
\end{align}
\subsubsection{Analysis of $\boldsymbol{\gamma_{n,m,\boldsymbol{K}}(s)}$}
First, we present the following cancellation law:
$$
\frac{\Gamma_{n,\mathbb{C}}(\frac{s+k}{2})}{\Gamma_{n,\mathbb{C}}(\frac{s}{2})}
=\prod_{j=0}^{\frac{k-2}{2}}\frac{\Gamma(\frac{s}{2}+\frac{k}{2}-j)}
{\Gamma(\frac{s}{2}+\frac{k}{2}-n-j)}\qquad (n>k).
$$
Hence, it is clear that
$$
\gamma_{n,m,\boldsymbol{K}}(s)
=\prod_{j=0}^{\frac{n-m-2}{2}}
\frac{\Gamma(n-s-j)}{\Gamma(-s-j)}\,
                                          \frac{\Gamma(s-m-j)}{\Gamma(s+n-m-j)}.
$$
Next, we use the following decomposition of $\gamma_{n,m,\boldsymbol{K}}(s)$:
\begin{align}
\label{Decgam}
& \gamma_{n,m,\boldsymbol{K}}(s) =
\gamma_{n,m,\boldsymbol{K}}^{(\text{I})}(s)\cdot \gamma_{n,m,\boldsymbol{K}}^{(\text{II})}(s)\\
 &\quad  \gamma_{n,m,\boldsymbol{K}}^{(\text{I})}(s):=\prod_{j=0}^{\frac{n-m-2}{2}}
                                                             \frac{\Gamma(s-m-j)}{\Gamma(-s-j)},\quad
   \gamma_{n,m,\boldsymbol{K}}^{(\text{II})}(s):=\prod_{j=0}^{\frac{n-m-2}{2}}
                                                             \frac{\Gamma(n-s-j)}{\Gamma(s+n-m-j)}.
                                                             \nonumber
\end{align}
Direct calculation shows that
\begin{align*}
& \lim_{s\to m}
\gamma_{n,m,\boldsymbol{K}}^{(\text{I})}(s)
 =\prod_{j=0}^{\frac{n-m-2}{2}}\left(\lim_{t\to 0}\frac{\Gamma(t-j)}{\Gamma(-t-m-j)}\right)
=(-1)^{m+1}\prod_{j=0}^{\frac{n-m-2}{2}}\frac{(m+j)!}{j\,!},\\
& ({\rm cf.} (\ref{Id}))\quad {\rm and}\\
& \gamma_{n,m,\boldsymbol{K}}^{(\text{II})}(m)
=
\prod_{j=0}^{\frac{n-m-2}{2}}
\frac{\Gamma(n-m-j)}{\Gamma(n-j)}
=\frac{\Gamma_{\frac{n-m}{2},\mathbb{C}}(n-m)}{\Gamma_{\frac{n-m}{2},\mathbb{C}}(n)}.
\end{align*}
Therefore,
$\gamma_{n,m,\boldsymbol{K}}(s) =
\gamma_{n,m,\boldsymbol{K}}^{(\text{I})}(s)\cdot \gamma_{n,m,\boldsymbol{K}}^{(\text{II})}(s)$
is holomorphic at $s=m$, and 
\begin{align*}
\lim_{s\to m}
\gamma_{n,m,\boldsymbol{K}}(s) &= \lim_{s\to m} \gamma_{n,m,\boldsymbol{K}}^{(\text{I})}(s)\cdot
                                          \gamma_{n,m,\boldsymbol{K}}^{(\text{II})}(m)\\
                                         &=(-1)^{m+1} \prod_{j=0}^{\frac{n-m-2}{2}}\frac{(m+j)!}{j\,!}\cdot
                          \frac{\Gamma_{\frac{n-m}{2},\mathbb{C}}(n-m)}{\Gamma_{\frac{n-m}{2},\mathbb{C}}(n)}.
\end{align*}
\subsubsection{Analysis of $\boldsymbol{\xi_{n,m,\boldsymbol{K}}(s)}$}
\label{analyxi}
We recall
$$
\xi_{n,m,\boldsymbol{K}}(s):=\prod_{j=0}^{n-1}
                                                \frac{\xi(n+m-2s-j;\chi_{\boldsymbol{K}}^j)}
                                                {\xi(2s+n-m-j;\chi_{\boldsymbol{K}}^j)},
$$
and decompose this as
\begin{align*}
& \xi_{n,m,\boldsymbol{K}}(s) = \rho_{n,m,\boldsymbol{K}}(s)\cdot \xi(2m-2s),\\
  & \qquad \quad          \rho_{n,m,\boldsymbol{K}}(s)=
                \frac{\prod_{\substack{0\leq j\leq n-1 \\ j\ne n-m}}\xi(n+m-2s-j;\chi_{\boldsymbol{K}}^j)}
                            {\prod_{j=0}^{n-1}\xi(2s+n-m-j;\chi_{\boldsymbol{K}}^j)}.
\end{align*}
(It should be noted that $n-m$ is even.) Since $\xi(s;;\chi_{\boldsymbol{K}})$ is holomorphic at $s=1$,
$\rho_{n,m,\boldsymbol{K}}(s)$ is holomorphic at $s=m$ and
$$
\rho_{n,m,\boldsymbol{K}}(m)=
 \frac{\prod_{\substack{0\leq j\leq n-1 \\ j\ne n-m}}\xi(n-m-j;\chi_{\boldsymbol{K}}^j)}
                            {\prod_{j=0}^{n-1}\xi(n+m-j;\chi_{\boldsymbol{K}}^j)}.
$$
Because $\mathop{{\rm Res}}_{s=m}\xi(2m-2s)=1/2$, we obtain
$$
\mathop{{\rm Res}}_{s=m}\xi_{n,m,\boldsymbol{K}}(s)=\tfrac{1}{2}\,\rho_{n,m,\boldsymbol{K}}(m).
$$
Consequently, $G_{n,m,\boldsymbol{K}}(s)$ has a simple pole
at $s=m$, and the residue is expressed as
$$
\mathop{{\rm Res}}_{s=m}G_{n,m,\boldsymbol{K}}(s)=
c_{n,m,\boldsymbol{K}}\in\mathbb{R}^\times.
$$
This proves (ii)
and consequently completes the proof of Theorem \ref{SU1}.
\vspace{2mm}
\\
The following theorem is an analogous result of Theorem \ref{ShimuraMain}
in the Siegel modular case, and this theorem specifies the modular form $f$, provided
by Shimura (cf. Theorem \ref{ShimuraMainH} ).
\begin{Thm} \textbf{(SU case)}
 \label{SU2}.
{\it We assume that $h_{\boldsymbol{K}}=1$ and
$n \equiv 1 \pmod{2}$.\\
Then, we have
$$
\mathop{{\rm Res}}_{s=1}E_{n-1,\boldsymbol{K}}^{(n)}(Z,s)
=\pi^{-n}\cdot c_{n,\boldsymbol{K}}\cdot E_{n-1,\boldsymbol{K}}^{(n)}(Z,0),
$$
where constant $c_{n,\boldsymbol{K}}$ is given by
\begin{equation*}
c_{n,\boldsymbol{K}}= 2^{-2n}\cdot D_{\boldsymbol{K}}^{\frac{n-1}{2}}\cdot n\cdot n!\,
                               \frac{B_{1,\chi_{\boldsymbol{K}}}}
                                       {B_{n,\chi_{\boldsymbol{K}}}\cdot B_{n+1}}
                               \in\mathbb{Q}^\times,
\end{equation*}
where $B_m$ {\rm (}resp. $B_{m,\chi}$ {\rm )} is the $m$-th Bernoulli
{\rm (}resp. generalized Bernoulli{\rm )}
number.
}
\end{Thm}
{\it Proof.}\quad 
This result means that constant $c_{n,1,\boldsymbol{K}}$ provided in Theorem \ref{SU1}
can be written as
$$
c_{n,1,\boldsymbol{K}}=\pi^{-n}\cdot c_{n,\boldsymbol{K}}
\quad (c_{n,\boldsymbol{K}}\in \mathbb{Q}^\times).
$$
We recall the decomposition, 
$$
G_{n,1,\boldsymbol{K}}(s) =\gamma_{n,1,\boldsymbol{K}}(s)\,\xi_{n,1,\boldsymbol{K}}(s).
\quad ({\rm cf.}\,  (\ref{DecH}))
$$
\subsubsection{Calculation of $\boldsymbol{\gamma_{n,1,\boldsymbol{K}}(s)}$}
We apply the decomposition shown in (\ref{Decgam}) for $m=1$.
\begin{align*}
& \gamma_{n,1,\boldsymbol{K}}(s) =
\gamma_{n,1,\boldsymbol{K}}^{(\text{I})}(s)\cdot \gamma_{n,1,\boldsymbol{K}}^{(\text{II})}(s)\\
 &\qquad  \gamma_{n,1,\boldsymbol{K}}^{(\text{I})}(s):=\prod_{j=0}^{\frac{n-3}{2}}
                                                             \frac{\Gamma(s-1-j)}{\Gamma(-s-j)},\quad
   \gamma_{n,1,\boldsymbol{K}}^{(\text{II})}(s):=\prod_{j=0}^{\frac{n-3}{2}}
                                                             \frac{\Gamma(n-s-j)}{\Gamma(s+n-1-j)}.
\end{align*}
Direct calculation shows that
\begin{align*}
& \lim_{s\to 1}\gamma_{n,1,\boldsymbol{K}}^{(\text{I})}(s)
=\prod_{j=0}^{\frac{n-3}{2}}\left( \lim_{s\to 1}  \frac{\Gamma(s-1-j)}{\Gamma(-s-j)} \right)
=\prod_{j=0}^{\frac{n-3}{2}}(j+1)=\left(\tfrac{n-1}{2}\right)!\,,\\
& \lim_{s\to 1}\gamma_{n,1,\boldsymbol{K}}^{(\text{II})}(s)
=\prod_{j=0}^{\frac{n-3}{2}}\frac{\Gamma(n-1-j)}{\Gamma(n-j)}
=\frac{\Gamma\left(\frac{n+1}{2}\right)}{\Gamma(n)}=\frac{\left(\frac{n-1}{2}\right)!}{(n-1)!}.
\end{align*}
Therefore, we obtain
\begin{equation}
\label{gamH}
\lim_{s\to 1} \gamma_{n,1,\boldsymbol{K}}(s)=
\frac{\left\{ \left( \frac{n-1}{2}\right)! \right\}^2}{(n-1)!}.
\end{equation}
\subsubsection{Calculation of $\boldsymbol{\xi_{n,1,\boldsymbol{K}}(s)}$}
We recall
$$
\xi_{n,1,\boldsymbol{K}}(s):=\prod_{j=0}^{n-1}
                                                \frac{\xi(n+1-2s-j;\chi_{\boldsymbol{K}}^j)}
                                                {\xi(2s+n-1-j;\chi_{\boldsymbol{K}}^j)},
$$
and decompose the right-hand side as
\begin{align*}
\xi_{n,1,\boldsymbol{K}}(s) &=\rho_{n,1,\boldsymbol{K}}(s)\,\xi(2-2s),
\vspace{2mm}
\\
& \rho_{n,1,\boldsymbol{K}}(s)=
\frac{\prod_{j=0}^{n-2} \xi(n+1-2s-j;\chi_{\boldsymbol{K}}^j)}
        {\prod_{j=0}^{n-1}\xi(2s+n-1-j;\chi_{\boldsymbol{K}}^j) }.
\end{align*}
From the general result for $\rho_{n,m,\boldsymbol{K}}(s)$ ($\S$ \ref{analyxi}), $\rho_{n,1,\boldsymbol{K}}(s)$ 
is holomorphic at $s=1$,
and $\xi(2-2s)$ has a simple pole $s=1$. That is,
$$
\mathop{{\rm Res}}_{s=1}\xi_{n,1,\boldsymbol{K}}(s)
= \rho_{n,1,\boldsymbol{K}}(1)\cdot \mathop{{\rm Res}}_{s=1}\xi(2-2s)
=\tfrac{1}{2}\rho_{n,1,\boldsymbol{K}}(1).
$$
Next, we calculate
$\rho_{n,1,\boldsymbol{K}}(1)$. By definition, we can write
$$
\rho_{n,1,\boldsymbol{K}}(1)
=\frac{\prod_{j=0}^{n-2} \xi(n-1-j;\chi_{\boldsymbol{K}}^j)}
        {\prod_{j=0}^{n-1}\xi(n+1-j;\chi_{\boldsymbol{K}}^j) }
=\frac{\xi(1;\chi_{\boldsymbol{K}})}{\xi(n+1)\,\xi(n;\chi_{\boldsymbol{K}})}.
$$
Furthermore,
\begin{align*}
& \xi(1;\chi_{\boldsymbol{K}})=-\pi^{\frac{1}{2}}\,B_{1,\chi_{\boldsymbol{K}}},\\
& \xi(n;\chi_{\boldsymbol{K}})=(-1)^{\frac{n+1}{2}}D_{\boldsymbol{K}}^{\frac{1-n}{2}}
                                              2^{n-1}\pi^{\frac{n}{2}}\left(\tfrac{n-1}{2}\right)!\,
                                              \frac{B_{n,\chi_{\boldsymbol{K}}}}{n!},\\
& \xi(n+1)=(-1)^{\frac{n+3}{2}}2^n\pi^{\frac{n+1}{2}}\left(\tfrac{n-1}{2}\right)!\,
                 \frac{B_{n+1}}{(n+1)!},
\end{align*}
Therefore, we have
$$
\rho_{n,1,\boldsymbol{K}}(1)
=\pi^{-n}\cdot D_{\boldsymbol{K}}^{\frac{n-1}{2}}\cdot 2^{1-2n}
\frac{n!\cdot (n+1)!}{\left\{ \left( \frac{n-1}{2}\right)! \right\}^2}
\frac{B_{1,\chi_{\boldsymbol{K}}}}{B_{n,\chi_{\boldsymbol{K}}}\cdot B_{n+1}}.
$$
Consequently, we obtain
\begin{align*}
\mathop{{\rm Res}}_{s=1}\,G_{n,1,\boldsymbol{K}}(s)
&= \gamma_{n,1,\boldsymbol{K}}(1)\,\mathop{{\rm Res}}_{s=1} \xi_{n,1,\boldsymbol{K}}(s)\\
&= \gamma_{n,1,\boldsymbol{K}}(1)\cdot \rho_{n,1,\boldsymbol{K}}(1)\cdot
      \mathop{{\rm Res}}_{s=1}\,\xi(2-2s)\\
&=\pi^{-n}\cdot 2^{-2n}\cdot D_{\boldsymbol{K}}^{\frac{n-1}{2}}\cdot n\cdot (n+1)!
       \frac{B_{1,\chi_{\boldsymbol{K}}}}{B_{n,\chi_{\boldsymbol{K}}}\cdot B_{n+1}}.
\end{align*}
This completes the proof of Theorem \ref{SU2}.
\hfill $\square$
\vspace{4mm}
\\
\section{Results for other types of Eisenstein series}
In this section, we provide some results analogous to those obtained
in the previous sections for the Eisenstein series on a quaternion half-space and an exceptional domain.
Both proofs are based on the functional equation of the corresponding Eisenstein series.

\subsection{Eisenstein series on quaternion half-space}
In \cite{Kim}, Kim studied the Eisenstein series $E_n(k,s,Z)$, on a quaternion half-space
 (\cite{Kim}, p. 215).
\\
In this study, we adopt the following normalization:
$$
E_{k,\mathbb{H}}^{(n)}(Z,s):=\text{det}(Y)^sE_n(k,2s,Z).
$$
He proved the following theorem:
\begin{Thm} (Kim \cite{Kim}, Theorem B)
\label{TheoremB}
{\it 
If $n$ is odd,
then $E_{2n-2,\mathbb{H}}^{(n)}(Z,s)$ has a simple pole at $s=1$, and the residue
is $\pi^{-n}$ times a singular modular form of weight $2n-2$ with rational
Fourier coefficients.
}
\end{Thm}
This theorem is a quaternionic analogy of Shimura's results (Theorem \ref{ShimuraMain}, Theorem \ref{ShimuraMainH}).
A similar argument in the previous sections yields the following result.
\begin{Thm} \textbf{(Quaternionic case)}
{\it 
Assume that $n$ is odd.
The Eisenstein series
$E_{2n-2,\mathbb{H}}^{(n)}(Z,s)$ has a simple pole at $s=1$, and the residue is
$$
\mathop{{\rm Res}}_{s=1}E_{2n-2,\mathbb{H}}^{(n)}(Z,s)
=\pi^{-n}\cdot c_{n,\mathbb{H}}\cdot E_{2n-2,\mathbb{H}}^{(n)}(Z,0),
$$
where
$$
c_{n,\mathbb{H}}=   2^{\frac{3-7n}{2}}\cdot\frac{(2n)!}{(n-1)!}\cdot\frac{1}{B_{2n}}\,
          \prod_{i=1}^{\frac{n-1}{2}}\frac{1-2^{4i-2n}}{1-2^{4i-2n-2}}.
$$
Here, $B_m$ is the $m$-th Bernoulli number.
}
\end{Thm}
As in earlier cases, the theorem above is a consequence of the
functional equation of $E_{k,\mathbb{H}}^{(n)}(Z,s)$ provided by
Kim:
$$
E_{k,\mathbb{H}}^{(n)}(Z,s) =\frac{h_{n,k}(\kappa(n)-k-s)}{h_{n,k}(s)}\,
                                   E_{k,\mathbb{H}}^{(n)}(Z,\kappa(n)-k-s),
$$
 where
 \begin{align*}
 h_{n,k}(s) = &  2^{\frac{k+2s}{2}[\frac{n}{2}]}\prod_{i=0}^{n-1}\xi (2s+k-2i)
                    \cdot \prod_{i=0}^{n-1}\prod_{j=1}^{\frac{k}{2}+n-1-i}(s+k-2i-j) \\
                 & \cdot \prod_{i=0,\,\text{odd}}^{n-1}{(1-2^{2i-k-2s})},   
 \end{align*}
 and $\kappa(n)=2n-1$,  $\xi(s)=\pi^{-s/2}\Gamma(s/2)\zeta(s)$, as before. 
 \begin{Rem}
 \label{Kim}
 Kim proved the functional equation of $E_n(0,s,Z)$ in \cite{Kim}, Theorem C.
 Furthermore, he provided the
 functional equation for the general $E_n(k,s,Z)$ (\cite{Kim3}).
 \end{Rem}
\subsection{Eisenstein series on exceptional domain}
 In \cite{Baily}, Baily studied modular forms on an exceptional domain in
 $\mathbb{C}^{27}$. The origin of the word ``exceptional'' is that
the exceptional group of type $E_7$ acts on this domain.
 Furthermore, Baily defined the Eisenstein series on this
 domain and proved the rationality of the Fourier coefficients.
 Subsequently, Karel \cite{Karel} provided an explicit formula for the Fourier coefficients
 of the holomorphic Eisenstein series.

In \cite{Kim2}, Kim provided the functional equation of the Eisenstein series
$E_{k,s}(Z)$, on an exceptional domain (\cite{Kim2}, Theorem B).
\\
Similar to the previous section, we use the following normalization:
$$
E_{k,\mathbb{O}}(Z,s):=\text{det}(Y)^sE_{k,2s}(Z).
$$
The following theorem was proved by Kim.
\begin{Thm} (Kim \cite{Kim2},  p.198, \text{Remark, (ii)})
\label{Kimexceptional}
{\it 
The Eisenstein series $E_{8,\mathbb{O}}(Z,s)$ has a simple pole at $s=1$,
and the residue
is $\pi^{-3}$ times a singular modular form of weight $8$ with rational
Fourier coefficients.
}
\end{Thm}
This theorem is an analogy of Shimura's results (Theorem \ref{ShimuraMain}, Theorem \ref{ShimuraMainH})
for the exceptional domain case.
By a similar argument as those in the earlier sections, we obtain the following result.
\begin{Thm} \textbf{(Exceptional case)}
{\it 
The Eisenstein series $E_{8,\mathbb{O}}(Z,s)$
has a simple pole at $s=1$, and the residue is
$$
\mathop{{\rm Res}}_{s=1}E_{8,\mathbb{O}}(Z,s)
=\pi^{-3}\cdot \frac{6237}{640}\,E_{8,\mathbb{O}}(Z,0).
$$
}
\end{Thm}
It follows from \cite{Kim2} that $E_{8,\mathbb{O}}(Z,0)$ is an exceptional singular
modular form of weight 8 with rational Fourier coefficients.

As in the other cases, the theorem above arises from the
functional equation of $E_{k,\mathbb{O}}(Z,s)$ provided by Kim (in the case $k=0$, see Theorem B in \cite{Kim2}, 
for general $k$, see \cite{Kim3}).
$$
E_{k,\mathbb{O}}(Z,s) =\frac{A_k(9-k-s)}{A_k(s)}
                                   E_{k,\mathbb{O}}(Z,9-k-s),
$$
where
$$
 A_k(s):=\xi(k+2s)\xi(k+2s-4)\xi(k+2s-8)
 \prod_{i=\frac{k}{2}-4}^{k-1}(s+i)\cdot \prod_{i=\frac{k}{2}-6}^{k-5}(s+i)\cdot
                           \prod_{i=\frac{k}{2}-8}^{k-9}(s+i).
$$
where $\xi(s)=\pi^{-s/2}\Gamma(s/2)\zeta(s)$.
\section{Remark}
\label{fremark}

The results obtained in this study are formulated based on the Jordan algebra.

We denote by $d$ (resp. $r$) the dimension (resp. rank) of the Jordan
Algebra, $J$. Moreover, we set
$$
\kappa:=\frac{d}{r}-1.
$$
Subsequently, this value is $\frac{n-1}{2}$ in the Siegel modular case,
$n-1$ in the Hermitian modular case, $2n-2$ in the quaternion modular case,
and $8$ in the exceptional modular case.

Our results can be summarized as follows. Let $\Omega$ be a formally real
simple Jordan algebra.
We denote by $E_{\kappa}(Z,s)$ 
the Eisenstein series  of weight $\kappa$ defined
on the corresponding tube domain, $\mathcal{D}=J+i\,\Omega$,
($\Omega=\text{exp}J$). Then, $E_{\kappa}(Z,s)$
has a simple pole
at $s=1$ at the most, and the residue can be written as
$$
\mathop{{\rm Res}}_{s=1}E_{\kappa}(Z,s)=\pi^{-r}\cdot c\,E_{\kappa}(Z,0),
$$
where $c\in\mathbb{Q}$ and $E_{\kappa}(Z,0)$ is a modular form of
weight $\kappa$ with rational Fourier
coefficients.



\end{document}